\documentclass[10pt,leqno]{amsart}
\usepackage{amsmath}
\usepackage{amsfonts}
\usepackage{amssymb}
\usepackage{mathtools}
\usepackage{graphicx}
\usepackage{hyperref}

\allowdisplaybreaks[4]

\usepackage{xcolor}
\newtheorem{theorem}{Theorem}[section]
\newtheorem*{theorem*}{Theorem}
\newtheorem{lemma}[theorem]{Lemma}

\theoremstyle{definition}
\newtheorem{definition}[theorem]{Definition}
\newtheorem{example}[theorem]{Example}

\newtheorem{remark}{Remark}[section]

\newcommand{\Rtwo}{\mathring{R}}
\newcommand{\Ric}{\operatorname{Ric}}

\newcommand{\tr}{\operatorname{tr}}
\newcommand{\R}{\mathbb{R}}

\author[Cheng]{Haiqing Cheng}
\address{School of Mathematical Sciences, Soochow University, Suzhou, 215006, China}
\email{chq4523@163.com}

\author[Wang]{Kui Wang}
\address{School of Mathematical Sciences, Soochow University, Suzhou, 215006, China}
\email{kuiwang@suda.edu.cn}

\title[Sphere Theorems for $\frac{10}{3}$-Positivity]
{Sphere Theorems for $\frac{10}{3}$-Positive Curvature Operator
of the Second Kind}


\begin{document}

\begin{abstract}
In this paper, we prove that every connected closed Riemannian manifold of dimension $n\geq4$ with $\frac{10}{3}$-positive curvature operator of the second kind is diffeomorphic to a spherical space form. 
We also obtain a classification of connected closed Riemannian manifolds with $\frac{10}{3}$-nonnegative curvature operator of the second kind. 
The key step is to show that $\frac{10}{3}$-positivity implies the PIC1 condition.
Finally, we give an example showing that $\frac{10}{3}$ is optimal for deducing PIC1 from $\alpha$-positivity.
\end{abstract}

\maketitle

\section{Introduction}
The action of the Riemannian curvature tensor on symmetric two-tensors has a long
history. It appeared in K\"ahler geometry in the work of Calabi and Vesentini
\cite{CV60} on deformations of compact locally symmetric K\"ahler manifolds.
In the Riemannian setting, the same action arises naturally in the deformation
theory of Einstein metrics and in Bochner--Weitzenb\"ock formulas. 
Bourguignon and Karcher \cite{BK78} studied its pinching properties and computed its eigenvalues on important compact symmetric spaces, including complex and quaternionic projective spaces.

Let $(M^n, g)$ be a Riemannian manifold. Since the space $S^2(T_pM)$ of symmetric two-tensors is not irreducible under
the orthogonal group, one considers the restriction to the traceless part
$S_0^2(T_pM)$. 
Following Nishikawa \cite{Nishikawa86}, the induced self-adjoint operator on $S_0^2(T_pM)$ is called the curvature operator of the
second kind. Nishikawa \cite{Nishikawa86} conjectured that a closed Riemannian manifold with
positive curvature operator of the second kind is diffeomorphic to a spherical
space form, and that a closed manifold with nonnegative curvature operator of
the second kind is diffeomorphic to a Riemannian locally symmetric space.
% Earlier evidence included the real-homology-sphere theorem of Ogiue and
% Tachibana \cite{OT79} and rigidity results of Kashiwada \cite{Kashiwada93}.

The subject has developed rapidly in recent years. Cao, Gursky, and Tran
\cite{CGT23} proved the positive part of Nishikawa's conjecture under the
weaker assumption of two-positivity. Their key observation is that
two-positive curvature operator of the second kind implies PIC1, so that one
can apply Brendle's Ricci-flow convergence theorem. Li \cite{Li21}
subsequently weakened two-positivity to three-positivity and classified closed
manifolds with three-nonnegative curvature operator of the second kind. 
% Li also proved that $4\frac12$-positive curvature operator of the second kind
% implies positive isotropic curvature and positive Ricci curvature
% \cite{Li22JGA}. 
Li also proved that $4\frac12$-positive curvature operator of the second kind implies positive isotropic curvature, and that $(n+\frac{n-2}{n})$-positive curvature operator of the second kind implies positive Ricci curvature \cite{Li22JGA}.
Further developments include results for K\"ahler manifolds
\cite{Li22Kahler, Li22PAMS}, product manifolds \cite{Li22product}, dimension
three \cite{FLi24}, Betti numbers and holonomy \cite{NPW23, NPWW23}, as well as more
general cone conditions \cite{Li25}.

Rigidity questions for Einstein manifolds under eigenvalue and cone conditions on the curvature operator of the second kind have been
studied in \cite{DF24, CW25-2, CW26-1, FL25, CW26}. Related results for
manifolds with harmonic curvature were obtained in \cite{FLD26}.
More recently, Li and Petersen obtained further rigidity results
involving the curvature operator of the second kind and its shifted
variants \cite{LP26COSK}, while Cao and Li established an invariant
relation between the curvature operators of the first and second kind
in dimension four \cite{CL26}. Other related developments include
Tachibana-type rigidity results for several curvature operators
\cite{LP26Tachibana} and a characterization of complex projective
space via the Calabi curvature operator \cite{EHLNPSW26}.

Let $\lambda_1\leq\lambda_2\leq\cdots\leq\lambda_N$ be the eigenvalues of the curvature operator of the second kind at a point $p\in M$, where
\[
N=\dim S_0^2(T_pM)=\frac{(n-1)(n+2)}2.
\]
For $\alpha\in[1,N)$, write $\alpha=k+t$, where
$k=\lfloor\alpha\rfloor$ and $0\leq t<1$. The curvature operator of
the second kind is called $\alpha$-positive (respectively,
$\alpha$-nonnegative) if
\[
\lambda_1+\cdots+\lambda_k+t\lambda_{k+1}>0
\qquad
(\text{respectively, }\geq0).
\]
When $\alpha=N$, these notions are defined using the full trace
$\lambda_1+\cdots+\lambda_N$ and are equivalent to positive and
nonnegative scalar curvature, respectively; see
\cite[Proposition 4.1(1)]{Li21}.
In particular,
$\frac{10}{3}$-positivity means
\[
\lambda_1+\lambda_2+\lambda_3+\frac13\lambda_4>0.
\]
If the condition holds at every point $p\in M$, the manifold is said to be
$\alpha$-positive (respectively, $\alpha$-nonnegative).
% Let
% \[
% \lambda_1\leq\lambda_2\leq\cdots\leq\lambda_N
% \]
% be the eigenvalues of the curvature operator of the second kind, where
% \[
% N=\dim S_0^2(T_pM)=\frac{(n-1)(n+2)}2.
% \]
% For a real number $\alpha\in[1, N)$, write
% \[
% \alpha=k+t,
% \qquad
% k=\lfloor\alpha\rfloor,
% \qquad
% 0\leq t<1.
% \]
% The operator is called $\alpha$-positive if
% \[
% \lambda_1+\cdots+\lambda_k+t\lambda_{k+1}>0,
% \]
% and $\alpha$-nonnegative if the same expression is nonnegative. When
% $\alpha=N$, these notions are defined using the full trace
% $\lambda_1+\cdots+\lambda_N$. In particular,
% $\frac{10}{3}$-positivity means
% \[
% \lambda_1+\lambda_2+\lambda_3+\frac13\lambda_4>0.
% \]

Our main result is the following.

\begin{theorem}\label{thm:main}
Let $(M^n,g)$ be a connected closed Riemannian manifold of dimension
$n\geq4$.
\begin{enumerate}
\item
If the curvature operator of the second kind is $\frac{10}{3}$-positive, then $M$ is diffeomorphic to a spherical space form.

\item
If the curvature operator of the second kind is $\frac{10}{3}$-nonnegative, then either $(M,g)$ is flat or $M$ is diffeomorphic to a
spherical space form.
\end{enumerate}
\end{theorem}
The crucial point is the following algebraic lemma. Here PIC1 means that the
product curvature tensor on $V\times\mathbb R$ has positive isotropic
curvature, whereas NIC1 means that it has nonnegative isotropic curvature.
For the precise four-frame characterizations, we refer the reader to Definition \ref{def:PIC1}.
% The corresponding four-frame characterizations are recalled precisely in
% Definition \ref{def:PIC1}.

\begin{lemma}\label{lem:main-algebraic}
Let $R$ be an algebraic curvature tensor on an $n$-dimensional Euclidean vector
space $V$, where $n\geq4$.
\begin{enumerate}
\item If $R$ has $\frac{10}{3}$-positive curvature operator of the second
kind, then $R$ is PIC1.
\item If $R$ has $\frac{10}{3}$-nonnegative curvature operator of the second
kind, then $R$ is NIC1.
\end{enumerate}
\end{lemma}

We first recall the following monotonicity property: if
$1\leq\alpha<\beta\leq N$, then $\alpha$-positivity implies
$\beta$-positivity, and the same statement holds for nonnegativity. 
This follows immediately from the definitions of $\alpha$-positivity
and $\alpha$-nonnegativity; see also the case $\theta=0$ of
\cite[Proposition 2.9]{Li25}.
% This
% follows immediately from the definition of fractional positivity; see also
% the case $\theta=0$ of \cite[Proposition 2.9]{Li25}. 
In particular,
$\frac{10}{3}$-positivity and $\frac{10}{3}$-nonnegativity are weaker than
three-positivity and three-nonnegativity, respectively, but stronger than the
corresponding $4\frac12$-positive and $4\frac12$-nonnegative conditions.

\begin{remark}\label{rem:comparison}
We record the relation between our results and some earlier results.
\begin{enumerate}
\item
Li proved in \cite[Proposition 4.1(4)]{Li21} that three-positivity
implies PIC1 and that three-nonnegativity implies NIC1.
Lemma \ref{lem:main-algebraic} extends these implications to the
weaker $\frac{10}{3}$-positive and $\frac{10}{3}$-nonnegative
conditions, respectively.

\item
At the global level, Theorem \ref{thm:main} extends the results under
three-positivity \cite[Theorem 1.6]{Li21} and three-nonnegativity
\cite[Corollary, p. 1644]{NPW23} to their $\frac{10}{3}$ counterparts.
Our result is also related to
\cite[Theorems 1.3 and 1.4]{Li22JGA}. Li proved that the weaker
$4\frac12$-positive condition yields a homeomorphism to a spherical
space form, which can be upgraded to a diffeomorphism when $n=4$ or
$n\geq12$. Under the stronger $\frac{10}{3}$-positive condition, we
obtain a diffeomorphism in every dimension $n\geq4$, including
$5\leq n\leq11$. In the nonnegative case, our stronger condition
yields the more restrictive flat-or-spherical classification.

\end{enumerate}
\end{remark}
The paper is organized as follows. In Section 2, we recall the curvature
operator of the second kind,  $\alpha$-positivity, and the PIC1 and
NIC1 curvature conditions. 
In Section 3, we prove Lemma \ref{lem:main-algebraic}, derive Theorem \ref{thm:main} from existing classification and rigidity results, and give a sharpness example for the algebraic implication.

\section{Preliminaries}
In this section, we introduce the notation used in the paper and recall some basic facts about the curvature operator of the second kind. We also establish a weighted variational inequality and review the NIC1 and PIC1 curvature conditions.

\subsection{The curvature operator of the second kind}

Let $(V, g)$ be an $n$-dimensional Euclidean vector space, and
let $\{e_1, \ldots, e_n\}$ be an orthonormal basis of $V$ with respect to $g$. We identify $V$ with
$V^*$ by means of $g$. Let $S^2 (V)$ denote the space of symmetric two-tensors on $V$, and let $\Lambda^2(V)$ denote the space of two-forms on $V$. 
% Denote by $S^2 (V)$ and $\Lambda^2(V)$ the spaces of symmetric two-tensors and two-forms on $V$, respectively. 
For $e_i, e_j\in V$, the symmetric product $\odot$ and the wedge product $\wedge$ are defined by $e_i\odot e_j=e_i\otimes e_j+e_j\otimes e_i$ and
$e_i\wedge e_j=e_i\otimes e_j-e_j\otimes e_i$, respectively. 
We define inner products on $S^2(V)$ and $\Lambda^2(V)$ by
% We equip $S^2(V)$ and $\Lambda^2(V)$ with the inner products
\[
\langle A, B\rangle_{S^2(V)}=\tr(A^TB),
\qquad
\langle A, B\rangle_{\Lambda^2(V)}=\frac12\tr(A^TB).
\]
Consequently, $\left\{\frac1{\sqrt2}e_i\odot e_j:1\leq i<j\leq n\right\}
\cup
\left\{\frac12e_i\odot e_i:1\leq i\leq n\right\}$
% \[
% \left\{\frac1{\sqrt2}e_i\odot e_j:1\leq i<j\leq n\right\}
% \cup
% \left\{\frac12e_i\odot e_i:1\leq i\leq n\right\}
% \]
is an orthonormal basis of $S^2(V)$. In particular,
$|e_i\odot e_j|^2=2$ for $i\neq j$, whereas
$|e_i\odot e_i|^2=4$. The space $S^2(V)$ splits into $O(V)$-irreducible subspaces as
\[S^2(V)=S_0^2(V)\oplus\mathbb Rg,\]
where $S_0^2(V)=\{h\in S^2(V):\tr h=0\}$. Denote by $N$ the dimension of $S_0^2(V)$, that is 
% Let $S_0^2(V)=\{h\in S^2(V):\tr h=0\}$. Then
% $S^2(V)=S_0^2(V)\oplus\mathbb Rg$ and
\[
N:=\dim S_0^2(V)=\frac{(n-1)(n+2)}2.
\]

Let $R$ be an algebraic curvature tensor on $V$. We use
the convention $R_{ijkl}=R(e_i, e_j, e_k, e_l)$, with
$R_{ijij}>0$ on the round sphere. The curvature symmetries are
$R_{ijkl}=-R_{jikl}=-R_{ijlk}=R_{klij}$, and the first Bianchi identity
takes the form
\begin{equation}\label{eq:Bianchi}
R_{ijkl}+R_{iklj}+R_{iljk}=0.
\end{equation}
%We write $K_{ij}:=R_{ijij}$.
The algebraic curvature tensor $R$ induces a self-adjoint map
$\bar R:S^2(V)\to S^2(V)$ defined by
\begin{equation}\label{eq:Rbar-action}
\bar R(e_i\odot e_j)
:=\sum_{k, l=1}^nR_{iklj}e_k\odot e_l.
\end{equation}
Let $\pi:S^2(V)\to S_0^2(V)$ be the orthogonal projection given by
$\pi(h)=h-\frac{\tr h}{n}g$.

\begin{definition}\label{def:second-kind}
Following Nishikawa \cite{Nishikawa86}, we call the self-adjoint operator
\[
\mathring R:=\pi\circ\bar R|_{S_0^2(V)}:
S_0^2(V)\rightarrow S_0^2(V)
\]
the curvature operator of the second kind.
% The curvature operator of the second kind is the self-adjoint operator
% \[
% \mathring R:=\pi\circ\overline R:
% S_0^2(V)\rightarrow S_0^2(V).
% \]
We use the same notation for its associated symmetric bilinear form, namely
$\mathring R(\varphi, \psi)
=\langle\mathring R\varphi, \psi\rangle$.
\end{definition}
If
$\varphi=\sum_{i, l=1}^n\varphi_{il}e_i\otimes e_l$ and
$\psi=\sum_{j, k=1}^n\psi_{jk}e_j\otimes e_k$ belong to $S_0^2(V)$, then
\begin{equation}\label{eq:R-component}
\mathring R(\varphi, \psi)
=\sum_{i, j, k, l=1}^n
R_{ijkl}\varphi_{il}\psi_{jk}.
\end{equation}
Indeed, the projection does not affect the pairing with $\psi$ because
$\psi$ is orthogonal to $g$.

Let $\lambda_1\leq\cdots\leq\lambda_N$ be the eigenvalues of
$\mathring R$. For $\alpha=k+t\in[1, N)$, where
$k=\lfloor\alpha\rfloor$ and $0\leq t<1$, define
\begin{equation}\label{eq:fractional-sum}
\sum_{j=1}^{\alpha}\lambda_j
:=\lambda_1+\cdots+\lambda_k+t\lambda_{k+1}.
\end{equation}
The operator $\mathring R$ is called $\alpha$-positive or
$\alpha$-nonnegative if the quantity in
\eqref{eq:fractional-sum} is positive or nonnegative, respectively.
For $\alpha=N$, this quantity is understood to be
$\sum_{j=1}^N\lambda_j$. When $\alpha=k$ is an integer, this reduces
to the usual $k$-positivity. In particular, $\frac{10}{3}$-positivity (respectively, $\frac{10}{3}$-nonnegativity) means
\[
\lambda_1+\lambda_2+\lambda_3+\frac13\lambda_4>0 \quad (\text{respectively}, \ge 0).
\]

We shall use the following weighted variational inequality.

\begin{lemma}\label{lem:ordered-quadruple}
Let $\xi_1, \xi_2, \xi_3, \xi_4\in S_0^2(V)$ be mutually orthogonal unit
tensors. Then
\begin{equation}\label{eq:ordered-quadruple}
\sum_{a=1}^3\mathring R(\xi_a, \xi_a)
+\frac13\mathring R(\xi_4, \xi_4)
\geq
\lambda_1+\lambda_2+\lambda_3+\frac13\lambda_4.
\end{equation}
\end{lemma}

\begin{proof}
Choose an orthonormal eigenbasis
$\{\omega_1, \ldots, \omega_N\}$ of $S_0^2(V)$ satisfying
$\mathring R\omega_i=\lambda_i\omega_i$, and set
\[
p_i:=\sum_{a=1}^3\langle\xi_a, \omega_i\rangle^2
+\frac13\langle\xi_4, \omega_i\rangle^2.
\]
Since $\xi_1, \ldots, \xi_4$ are orthonormal, Bessel's inequality gives   
\[
0\leq p_i
\leq\sum_{a=1}^4\langle\xi_a, \omega_i\rangle^2
\leq1.
\]
By Parseval's identity, $\sum_{i=1}^Np_i=3+\frac13=\frac{10}{3}$.
Moreover,
\[
\sum_{a=1}^3\mathring R(\xi_a, \xi_a)
+\frac13\mathring R(\xi_4, \xi_4)
=\sum_{i=1}^N\lambda_ip_i.
\]
Because $\lambda_1\leq\cdots\leq\lambda_N$, this weighted sum is
minimized, under the conditions $0\leq p_i\leq1$ and
$\sum_{i=1}^{N} p_i=\frac{10}{3}$, by taking
\[
p_1=p_2=p_3=1,\qquad p_4=\frac13,
\qquad p_i=0\quad(i\geq5).
\]
This proves \eqref{eq:ordered-quadruple}.
\end{proof}

\subsection{The PIC1 and NIC1 curvature conditions}

For the definitions of NIC1 and PIC1, we follow \cite[Proposition 4]{Brendle08}, \cite[Definition 2.4]{CGT23}, and \cite[Definition 4.1]{Li21}.
% We use the following convention; see
% \cite[Proposition 4]{Brendle08},
% \cite[Definition 2.4]{CGT23}, and
% \cite[Definition 4.1]{Li21}.

\begin{definition}\label{def:PIC1}
Let $R$ be an algebraic curvature tensor on a Euclidean vector
space $V$ with $\dim V\geq4$. We say that $R$ satisfies \emph{NIC1}
if
\begin{equation}\label{eq:PIC1}
R_{1313}+\lambda^2R_{1414}
+R_{2323}+\lambda^2R_{2424}
-2\lambda R_{1234}\geq0
\end{equation}
for every orthonormal four-frame
$\{e_1, e_2, e_3, e_4\}\subset V$ and every
$\lambda\in[-1, 1]$. If the inequality is strict for every such
four-frame and every $\lambda\in[-1, 1]$, then $R$ satisfies
\emph{PIC1}.
\end{definition}

The conditions NIC1 and PIC1 are also called weakly PIC1 and strictly
PIC1, respectively. By \cite[Proposition 4]{Brendle08}, NIC1 is
equivalent to requiring that the induced curvature tensor on
$V\times\mathbb R$ have nonnegative isotropic curvature. The strict
version gives the corresponding equivalence between PIC1 and positive
isotropic curvature on $V\times\mathbb R$.

Taking $\lambda=1$ in \eqref{eq:PIC1} shows that NIC1 implies
nonnegative isotropic curvature on $V$.
NIC1 also implies nonnegative Ricci curvature; see
\cite[Section 2]{Topping24}.
Indeed, suppose that $R$ is NIC1. Fix a unit vector $v$ and choose
an orthonormal basis $\{e_1, \ldots, e_{n-1}\}$ of $v^\perp$.
For each $i<j$, extend $e_i, e_j, v$ to an orthonormal four-frame.
Taking $\lambda=0$ in \eqref{eq:PIC1} gives
\[
R(e_i, v, e_i, v)+R(e_j, v, e_j, v)\geq0.
\]
Summing over all pairs $i<j$, we obtain
\[
(n-2)\Ric(v, v)\geq0.
\]
Thus $\Ric\geq0$. The same argument with strict inequalities
shows that PIC1 implies $\Ric>0$.

% We shall use two global consequences of these curvature conditions.
% By Brendle's convergence theorem \cite[Theorem 2]{Brendle08}, if a closed manifold of dimension at least four has an initial metric satisfying PIC1, then its normalized Ricci flow converges to a metric of constant positive sectional curvature.

% By Brendle's convergence theorem
% \cite[Theorem 2]{Brendle08}, the normalized Ricci flow on a closed
% manifold of dimension at least four with a PIC1 initial metric
% converges to a metric of constant positive sectional curvature. 
% We shall also use the classification of closed simply connected
% irreducible NIC1 manifolds
% \cite[Theorem 9.33]{Brendle10}; see also
% \cite[Section 7]{Li21}. Such a manifold is diffeomorphic to a sphere,
% is K\"ahler, or is isometric to a compact irreducible symmetric
% space.

% The second consequence is the classification of closed simply connected irreducible NIC1 manifolds \cite[Theorem 9.33]{Brendle10}, see also \cite[Section 7]{Li21}. Such a manifold is either diffeomorphic to a sphere, Kähler, or isometric to a compact irreducible symmetric space.

We shall use two global consequences of these curvature conditions. The first is Brendle's convergence theorem \cite[Theorem 2]{Brendle08}, which says that if a closed manifold of dimension at least four has an initial metric satisfying PIC1, then its normalized Ricci flow converges to a metric of constant positive sectional curvature. The second is the classification of closed simply connected irreducible NIC1 manifolds \cite[Theorem 9.33]{Brendle10}; see also \cite[Section 7]{Li21}. Such a manifold is diffeomorphic to a sphere, is K\"ahler, or is
isometric to a compact irreducible symmetric space.

\section{Proof of the main theorem}

In this section, we first prove Lemma \ref{lem:main-algebraic}, then establish
Theorem \ref{thm:main}, and finally provide an example showing that the constant $\frac{10}{3}$ is optimal for deducing PIC1 from
$\alpha$-positivity.

\subsection{Proof of Lemma \ref{lem:main-algebraic}}

% We now prove Lemma \ref{lem:main-algebraic}. Our argument is based on the
% same general idea as \cite[Theorem 1.4]{CGT23} and
% \cite[Lemma 4.1]{Li21}: we construct suitable orthonormal traceless
% symmetric two-tensors and apply the curvature condition to the resulting
% curvature sums. 
% However, the coefficient $\frac13$ in the
% $\frac{10}{3}$-positivity condition requires a more delicate choice of
% tensors and a careful matching of coefficients.
We now prove Lemma \ref{lem:main-algebraic}. The proof follows the same general strategy as in \cite[Theorem 1.4]{CGT23} and \cite[Lemma 4.1]{Li21}. We construct a suitable orthonormal family of traceless symmetric two-tensors and apply the curvature condition to the resulting
curvature sums. 
The new difficulty is the coefficient $\frac13$ in the $\frac{10}{3}$-positivity assumption, which forces a more delicate choice of tensors and a careful matching of coefficients.
\begin{proof}[Proof of Lemma \ref{lem:main-algebraic}]
Fix an orthonormal four-frame $\{e_1, e_2, e_3, e_4\}$ and extend it to an orthonormal basis of $V$. 
For $\lambda\in\mathbb R$, set
\[
\begin{aligned}
I(e_1, e_2, e_3, e_4; \lambda)
:={}&R_{1313}+\lambda^2R_{1414}
+R_{2323}+\lambda^2R_{2424}
-2\lambda R_{1234}.
\end{aligned}
\]
Since
\begin{equation}\label{eq:change-sign-lambda}
I(e_1, e_2, e_3, e_4; \lambda)
=
I(e_1, e_2, e_3, -e_4; -\lambda),
\end{equation}
it suffices to consider $\lambda\in[0,1]$, which we assume throughout
the remainder of the proof.

For $1\leq i, j\leq n$, set
\[
E_{ij}:=\frac{1}{\sqrt{2}}e_i\odot e_j.
\]
Notice that $E_{ij}\in S_0^2(V)$ and $|E_{ij}|=1$ whenever
$i\neq j$. We also define
\[
H_1:=\frac{1}{\sqrt{6}}
\left(
e_1\otimes e_1+e_2\otimes e_2-2e_3\otimes e_3
\right).
\]
We first record the curvature identities needed in the proof. By \eqref{eq:R-component}, if $i\neq j$ and $k\neq l$, then
\[
\Rtwo(E_{ij}, E_{kl})
=R_{iklj}+R_{ilkj}.
\]
In particular, if $i\neq j$, then
\begin{equation}\label{eq:Eij-sectional}
\Rtwo(E_{ij}, E_{ij})=R_{ijij}.
\end{equation}
Moreover, the first Bianchi identity gives
\begin{align}\label{eq:mixed-bianchi}
\begin{aligned}
\Rtwo(E_{13}, E_{24}) - \Rtwo(E_{14}, E_{23})
&= (R_{1243}+R_{1423}) - (R_{1234}+R_{1324})\\
&=-2R_{1234}+R_{1342}+R_{1423} 
=-3R_{1234}.
\end{aligned}
\end{align}
For the diagonal tensor $H_1$, a direct calculation gives
\begin{align*}
\Rtwo(H_1, H_1)
&=-2\left(
\frac16R_{1212}
-\frac13R_{1313}
-\frac13R_{2323}
\right)
=-\frac13R_{1212}
+\frac23R_{1313}
+\frac23R_{2323}.
\end{align*}
Together with $\Rtwo(E_{12}, E_{12})=R_{1212}$,  this gives
\begin{equation}\label{eq:H1-E12-cancellation}
\Rtwo(H_1, H_1)
+\frac13\Rtwo(E_{12}, E_{12})
=
\frac23(R_{1313}+R_{2323}).
\end{equation}

We now introduce the coefficients that will be determined later. Let
$a, b\geq0$ satisfy $a^2+b^2=1$,
and define
\[
\begin{aligned}
\varphi_1&=aE_{13}+bE_{24},
&
\psi_1&=-bE_{13}+aE_{24},\\
\varphi_2&=bE_{14}-aE_{23},
&
\psi_2&=aE_{14}+bE_{23}.
\end{aligned}
\]
Since $a^2+b^2=1$, the six tensors
\[
\varphi_1, \quad \varphi_2, \quad
\psi_1, \quad \psi_2, \quad H_1, \quad E_{12}
\]
are mutually orthogonal unit tensors in $S_0^2(V)$.
Let $\beta, d\geq0$ be parameters to be determined. By bilinearity and
\eqref{eq:Eij-sectional}, we have
\begin{align*}
\Rtwo(\varphi_1, \varphi_1)
&=a^2R_{1313}+b^2R_{2424}
+2ab\Rtwo(E_{13}, E_{24}),\\
\Rtwo(\psi_1, \psi_1)
&=b^2R_{1313}+a^2R_{2424}
-2ab\Rtwo(E_{13}, E_{24}),\\
\Rtwo(\varphi_2, \varphi_2)
&=b^2R_{1414}+a^2R_{2323}
-2ab\Rtwo(E_{14}, E_{23}),\\
\Rtwo(\psi_2, \psi_2)
&=a^2R_{1414}+b^2R_{2323}
+2ab\Rtwo(E_{14}, E_{23}).
\end{align*}
Combining these four identities and using
\eqref{eq:mixed-bianchi}, we obtain
\begin{align*}
&\Rtwo(\varphi_1, \varphi_1)
+\Rtwo(\varphi_2, \varphi_2)
+\beta\Rtwo(\psi_1, \psi_1)
+\beta\Rtwo(\psi_2, \psi_2)\\
&=(a^2+\beta b^2)(R_{1313}+R_{2323}) 
+ (b^2+\beta a^2)(R_{1414}+R_{2424}) -6(1-\beta)abR_{1234}.
\end{align*}
Adding 
$d$ times \eqref{eq:H1-E12-cancellation} to the above identity, we obtain
\begin{align}
\begin{aligned}
&\Rtwo(\varphi_1, \varphi_1)
+\Rtwo(\varphi_2, \varphi_2)
+\beta\Rtwo(\psi_1, \psi_1)
+\beta\Rtwo(\psi_2, \psi_2)\\
&\quad +d\Rtwo(H_1, H_1)
+\frac d3\Rtwo(E_{12}, E_{12})\\
&=
\left(a^2+\beta b^2+\frac{2d}{3}\right)
(R_{1313}+R_{2323})
+(b^2+\beta a^2)(R_{1414}+R_{2424})\\
&\quad-6(1-\beta)abR_{1234}.  
\end{aligned}
\label{eq:undetermined-curvature-sum}
\end{align}

Set $\Lambda = \lambda_1+\lambda_2+\lambda_3+\frac13\lambda_4$.
Applying Lemma \ref{lem:ordered-quadruple} successively to
\[
(\varphi_1, \varphi_2, \psi_1, \psi_2), \qquad
(\varphi_1, \varphi_2, \psi_2, \psi_1), \qquad
(\varphi_1, \varphi_2, H_1, E_{12}),
\]
we obtain
\begin{align}\label{eq:ordered-sum}
\begin{aligned}
\Rtwo(\varphi_1, \varphi_1)
+\Rtwo(\varphi_2, \varphi_2)
+\Rtwo(\psi_1, \psi_1)
+\frac13\Rtwo(\psi_2, \psi_2)
&\geq\Lambda,\\
\Rtwo(\varphi_1, \varphi_1)
+\Rtwo(\varphi_2, \varphi_2)
+\Rtwo(\psi_2, \psi_2)
+\frac13\Rtwo(\psi_1, \psi_1)
&\geq\Lambda,\\
\Rtwo(\varphi_1, \varphi_1)
+\Rtwo(\varphi_2, \varphi_2)
+\Rtwo(H_1, H_1)
+\frac13\Rtwo(E_{12}, E_{12})
&\geq\Lambda.
\end{aligned}
\end{align}
Suppose that $d+\frac32\beta=1$.
% \begin{equation}\label{eq:beta-d-relation}
% d+\frac32\beta=1.
% \end{equation}
Multiplying the first two inequalities in \eqref{eq:ordered-sum} by $\frac{3\beta}{4}$, multiplying the third
inequality by $d$, and then adding them, we obtain
\begin{align}\label{eq:weighted-lower-bound}
\begin{aligned}
&\Rtwo(\varphi_1, \varphi_1)
+\Rtwo(\varphi_2, \varphi_2)
+\beta\Rtwo(\psi_1, \psi_1)
+\beta\Rtwo(\psi_2, \psi_2)\\
&\quad
+d\Rtwo(H_1, H_1)
+\frac d3\Rtwo(E_{12}, E_{12})
\geq\Lambda.
\end{aligned}
\end{align}
Here the coefficient of $\Lambda$ on the
right-hand side is $\frac{3\beta}{4}+\frac{3\beta}{4}+d
=\frac32\beta+d=1$.

It remains to determine $a, b, \beta$, and $d$. We require the right-hand
side of \eqref{eq:undetermined-curvature-sum} to be
\[
c\,I(e_1, e_2, e_3, e_4; \lambda)=
c\left(
R_{1313}+\lambda^2R_{1414}
+R_{2323}+\lambda^2R_{2424}
-2\lambda R_{1234}
\right)
\]
for some $c>0$. Comparing the coefficients gives
\begin{equation}\label{eq:coefficient-system}
\left\{
\begin{aligned}
a^2+\beta b^2+\frac{2d}{3}&=c,\\
b^2+\beta a^2&=c\lambda^2,\\
(1-\beta)ab&=\frac{c\lambda}{3},
\end{aligned}
\right.
\end{equation}
together with $a^2+b^2=1$ and  $d+\frac32\beta=1$.
Using $a^2+b^2=1$, the second equation in \eqref{eq:coefficient-system} can be rewritten as
\begin{equation}\label{eq:beta-a-relation}
(1-\beta)a^2=1-c\lambda^2.
\end{equation}
Moreover, the left-hand side of the first equation in \eqref{eq:coefficient-system} becomes
\begin{align*}
a^2+\beta b^2+\frac{2d}{3}
=a^2+\beta(1-a^2)+\frac23-\beta=\frac23+(1-\beta)a^2 =\frac53-c\lambda^2,
\end{align*}
where the last equality follows from
\eqref{eq:beta-a-relation}. Hence
\begin{equation}\label{eq:c-value}
c=\frac5{3(1+\lambda^2)}.
\end{equation}
One can verify that
\begin{equation}\label{eq:one-minus-c-lambda}
1-c\lambda^2
=
\frac{3-2\lambda^2}{3(1+\lambda^2)}>0.
\end{equation}
Dividing the third equation in \eqref{eq:coefficient-system} by \eqref{eq:beta-a-relation}, we obtain $\frac ba = \frac{c\lambda}{3(1-c\lambda^2)}$.
Combining this with $a^2+b^2=1$, and using \eqref{eq:c-value}, we get
\begin{align*}
a^2
&=
\frac{9(1-c\lambda^2)^2}
{9(1-c\lambda^2)^2+c^2\lambda^2}
=
\frac{9(3-2\lambda^2)^2}
{9(3-2\lambda^2)^2+25\lambda^2}, \\
b^2
&=
\frac{c^2\lambda^2}
{9(1-c\lambda^2)^2+c^2\lambda^2}
=
\frac{25\lambda^2}
{9(3-2\lambda^2)^2+25\lambda^2}.
\end{align*}
Since $\lambda\in[0, 1]$, we may take
\begin{equation}\label{eq:ab-values}
a=
\frac{3(3-2\lambda^2)}
{\sqrt{9(3-2\lambda^2)^2+25\lambda^2}},
\qquad
b=
\frac{5\lambda}
{\sqrt{9(3-2\lambda^2)^2+25\lambda^2}}.
\end{equation}
It remains to determine $\beta$ and $d$. By
\eqref{eq:beta-a-relation}, \eqref{eq:one-minus-c-lambda}, and
\eqref{eq:ab-values},
\begin{align*}
1-\beta
=\frac{1-c\lambda^2}{a^2}=
\frac{9(3-2\lambda^2)^2+25\lambda^2}
{27(1+\lambda^2)(3-2\lambda^2)}.
\end{align*}
It follows that
\begin{equation*}%\label{eq:beta-value}
\beta
=
\frac{10\lambda^2(11-9\lambda^2)}
{27(1+\lambda^2)(3-2\lambda^2)}, 
\qquad 
d=1-\frac32\beta=\frac{27\lambda^4-46\lambda^2+27}
{9(1+\lambda^2)(3-2\lambda^2)}.
\end{equation*}
% Finally, using $d=1-\frac32\beta$, we obtain
% \begin{equation*}\label{eq:d-value}
% d
% =
% \frac{27\lambda^4-46\lambda^2+27}
% {9(1+\lambda^2)(3-2\lambda^2)}.
% \end{equation*}
For $\lambda\in[0, 1]$, a direct calculation shows that all the
denominators above are positive and that $\beta\geq0, d>0$.
Thus, all the weights used above are nonnegative.

Substituting these coefficients into
\eqref{eq:undetermined-curvature-sum}, we obtain
\begin{align}\label{eq:final-PIC1-identity}
\begin{aligned}
&\Rtwo(\varphi_1, \varphi_1)
+\Rtwo(\varphi_2, \varphi_2)
+\beta\Rtwo(\psi_1, \psi_1)
+\beta\Rtwo(\psi_2, \psi_2)\\
&\quad 
+d\Rtwo(H_1, H_1)
+\frac d3\Rtwo(E_{12}, E_{12})
=
\frac5{3(1+\lambda^2)}
I(e_1, e_2, e_3, e_4; \lambda).
\end{aligned}
\end{align}
Combining \eqref{eq:weighted-lower-bound} with
\eqref{eq:final-PIC1-identity} gives
\begin{equation}\label{eq:PIC1-lower-bound}
\frac5{3(1+\lambda^2)}
I(e_1, e_2, e_3, e_4; \lambda)
\geq\Lambda, \quad \lambda \in [0, 1].
\end{equation}
By \eqref{eq:change-sign-lambda}, the same inequality holds for every
$\lambda\in[-1, 1]$. Since $\Lambda>0$ in the
$\frac{10}{3}$-positive case and $\Lambda\geq0$ in the
$\frac{10}{3}$-nonnegative case, $R$ is PIC1 in the first case and
NIC1 in the second. 
\end{proof}

\begin{remark}\label{rem:lambda-range}
Although Definition \ref{def:PIC1} uses
$\lambda\in[-1,1]$, the curvature inequalities obtained in
Lemma \ref{lem:main-algebraic} hold for every $\lambda\in\mathbb R$.
Indeed, for $\lambda\neq0$,
\[
I(e_1, e_2, e_3, e_4;\lambda)
=
\lambda^2
I\left(e_1, e_2, e_4, -e_3;\frac1\lambda\right).
\]
Thus, the parameter range $[-1, 1]$ in these inequalities may be
replaced by $\mathbb R$.
\end{remark}

\subsection{Proof of Theorem \ref{thm:main}}

The positive case follows from Lemma \ref{lem:main-algebraic} and
Brendle's convergence theorem, as in \cite[Theorem 1.2]{CGT23} and
\cite[Theorem 1.6]{Li21}. For the nonnegative case, we follow
\cite[Theorem 1.7]{Li21} and use \cite[Theorem A]{NPWW23} for the
symmetric-space case.

\begin{proof}[Proof of Theorem \ref{thm:main}]
(1) Suppose first that $(M, g)$ has $\frac{10}{3}$-positive curvature
operator of the second kind. By Lemma \ref{lem:main-algebraic},
$(M, g)$ is PIC1. Brendle's convergence theorem
\cite[Theorem 2]{Brendle08} then implies that $M$ is diffeomorphic to a
spherical space form; see also \cite{CGT23, Li21}.

(2) 
Suppose that $(M, g)$ is nonflat, and let $(\widetilde M, \widetilde g)$ be its universal Riemannian cover.
Then $(\widetilde M, \widetilde g)$ is also nonflat and has
$\frac{10}{3}$-nonnegative curvature operator of the second kind. By Lemma \ref{lem:main-algebraic}, it is NIC1 and hence has nonnegative Ricci curvature.
The Cheeger--Gromoll splitting theorem \cite{CG71} implies $\widetilde M$ is isometric to a product of the form $N_0^{n-k}\times \R^k$, where $N_0$ is compact. 
The monotonicity of the $\alpha$-nonnegative condition shows that
$\widetilde M$ has $n$-nonnegative curvature operator of the second kind. It follows from \cite[Theorem 1.8]{Li21} that $\widetilde M$ is locally irreducible. Hence the above splitting
forces $k=0$, so $\widetilde M$ is compact and irreducible.

% By the monotonicity of the $\alpha$-nonnegative condition,
% $\widetilde M$ has $n$-nonnegative curvature operator of the second
% kind.
% Due to \cite[Theorem 1.8]{Li21}, $\widetilde M$ is locally irreducible, which implies $k=0$. Hence
% $\widetilde M$ is compact and irreducible.
%Since $\widetilde M$ is complete and simply connected, the de Rham decomposition theorem shows that its local irreducibility implies its irreducibility; see \cite[Chapter~IV, Theorem~6.2]{KN63}.

By the classification of closed simply connected irreducible NIC1
manifolds (see \cite[Theorem 9.33]{Brendle10} or
\cite[Section 7]{Li21}), one of the following holds:
$\widetilde M$ is diffeomorphic to $\mathbb S^n$, is K\"ahler, or is
isometric to a compact irreducible symmetric space.

We now consider the latter two possibilities.
The K\"ahler case is impossible. Indeed, by monotonicity, $\widetilde M$ has $4$-nonnegative curvature operator of the second kind and hence must be flat by \cite[Theorem 1.9]{Li21}, contradicting our assumption.
Suppose that $\widetilde M$ is a compact irreducible symmetric space.
Then it is locally symmetric and has $n$-nonnegative curvature
operator of the second kind. By \cite[Theorem A]{NPWW23},
$\widetilde M$ has constant sectional curvature, say $c$. In this
case, $\Rtwo=c\,\operatorname{Id}$ on
$S_0^2(T\widetilde M)$, so its $\frac{10}{3}$-nonnegativity implies
$\frac{10}{3}c\geq0$. Since $\widetilde M$ is nonflat, $c>0$.
Therefore, $\widetilde M$ is isometric, up to scaling, to the round sphere.

Consequently, the only nonflat possibility is that $M$ is
diffeomorphic to a spherical space form. This completes the proof.
\end{proof}

\subsection{Optimality of
\texorpdfstring{$\frac{10}{3}$}{10/3} for deducing PIC1}

% We conclude with an example showing that no value
% $\alpha>\frac{10}{3}$ can replace $\frac{10}{3}$ in the implication
% from $\alpha$-positivity to PIC1.

We conclude with an example showing that $\frac{10}{3}$ is the largest
value of $\alpha$ for which $\alpha$-positivity implies PIC1 in general.

\begin{example}
For every $n\geq4$, there exists an algebraic curvature tensor $R$ on
$\mathbb R^n$ satisfying the following properties:
\begin{enumerate}
\item $\Rtwo$ is $\frac{10}{3}$-nonnegative and
$\alpha$-positive for every
$\alpha\in(\frac{10}{3}, N]$;
\item $R$ is NIC1 but is not PIC1.
\end{enumerate}
Thus, $\alpha$-positivity does not imply PIC1 in general for any
$\alpha>\frac{10}{3}$.
\end{example}

\begin{proof}
Let $V=\mathbb R^n$ and let $\{e_1, \ldots, e_n\}$ be its standard orthonormal basis. Define an algebraic curvature tensor $R$ on $V$ by prescribing
\begin{align}\label{eqn:RiemannCurv}
\begin{aligned}
R_{1212}&=1, \quad R_{1313}=R_{2323}=0,\\
R_{3a3a}&=1, \quad R_{1a1a}=R_{2a2a}=\frac43,
&&4\leq a\leq n,\\
R_{abab}&=1,
&&4\leq a<b\leq n,
\end{aligned}
\end{align}
and setting all mixed curvature components between distinct coordinate
2-planes equal to zero. We then extend $R$ by the curvature
symmetries. The first Bianchi identity follows immediately from this
diagonal construction.

For $1\leq i<j\leq n$, set
\[
E_{ij}:=\frac1{\sqrt2}e_i\odot e_j.
\]
Then $E_{ij}\in S_0^2(V)$ and $|E_{ij}|=1$. By
\eqref{eq:R-component}, for $i<j$ and $k<l$, we have
\[
\Rtwo(E_{ij}, E_{kl})=R_{iklj}+R_{ilkj}.
\]
By \eqref{eqn:RiemannCurv} and the definition of the remaining
curvature components, it follows that
\[
\Rtwo(E_{ij}, E_{kl})=0
\quad\text{if }(i, j)\neq(k, l),
\qquad
\Rtwo(E_{ij}, E_{ij})=R_{ijij}.
\]
Moreover, $\Rtwo(E_{ij}, h)=0$ for every diagonal traceless symmetric
tensor $h$. Since the off-diagonal tensors $E_{kl}$ and the diagonal
traceless tensors together span $S_0^2(V)$, each $E_{ij}$ is an
eigentensor of $\Rtwo$ with eigenvalue $R_{ijij}$.
More precisely, $E_{13}$ and $E_{23}$ have eigenvalue zero, 
$E_{1a}$ and $E_{2a}$ have eigenvalue $\frac43$ for
$4\leq a\leq n$, and
\[
E_{12},\qquad E_{3a}\ (4\leq a\leq n),\qquad
E_{ab}\ (4\leq a<b\leq n)
\]
have eigenvalue $1$.

We have determined all the eigenvalues of $\Rtwo$ on the off-diagonal
symmetric tensors. It remains to consider the diagonal traceless
tensors, which form the subspace
\[
\mathcal D_0
:=
\left\{
\sum_{i=1}^n x_i e_i\otimes e_i:
\sum_{i=1}^n x_i=0
\right\}.
\]
If $h=\sum_{i=1}^{n} x_i e_i\otimes e_i$ and
$k=\sum_{i=1}^{n} y_i e_i\otimes e_i$ belong to $\mathcal D_0$, then
\begin{equation}\label{eq:sharp-diagonal-bilinear}
\Rtwo(h, k)
=
-\sum_{1\le i<j\le n}R_{ijij}(x_iy_j+x_jy_i).
\end{equation}

Set
\[
q_0:=
\frac1{\sqrt6}
\left(
e_1\otimes e_1+e_2\otimes e_2-2e_3\otimes e_3
\right).
\]
For an arbitrary
$h=\sum_{i=1}^{n} x_i e_i\otimes e_i\in\mathcal D_0$,
\eqref{eq:sharp-diagonal-bilinear} and
\eqref{eqn:RiemannCurv} give
\[
\begin{aligned}
\Rtwo(q_0, h)
&=-\frac1{\sqrt6}\Big[
R_{1212}(x_1+x_2)
+R_{1313}(x_3-2x_1)
+R_{2323}(x_3-2x_2)\\
&\quad+\sum_{j=4}^n
\big(
R_{1j1j}+R_{2j2j}-2R_{3j3j}
\big)x_j
\Big]\\
&=
\frac1{\sqrt6}
\left(
-x_1-x_2-\frac23\sum_{j=4}^n x_j
\right).
\end{aligned}
\]
Using $\sum_{j=4}^n x_j=-(x_1+x_2+x_3)$, we obtain
\[
\Rtwo(q_0,h)
=
\frac1{\sqrt6}
\left(
-\frac13x_1-\frac13x_2+\frac23x_3
\right)
=
-\frac13\langle q_0, h\rangle.
\]
This holds for every $h\in\mathcal D_0$. Moreover, if $\zeta$ is an
off-diagonal symmetric tensor, then
$\Rtwo(q_0, \zeta)=0$ and $\langle q_0, \zeta\rangle=0$. Hence
\[
\langle\Rtwo q_0, \xi\rangle
=
-\frac13\langle q_0, \xi\rangle
\]
for every $\xi\in S_0^2(V)$. Therefore,
$\Rtwo q_0=-\frac13q_0$, which shows that $q_0$ is an eigentensor with eigenvalue
$-\frac13$.

Similarly, let
\[
q_1:=
\frac1{\sqrt2}
\left(
e_1\otimes e_1-e_2\otimes e_2
\right),
\qquad
\eta:=
\sum_{a=4}^n z_a e_a\otimes e_a \quad \text{with} 
\sum_{a=4}^n z_a=0.
\]
For any $h=\sum_{i=1}^{n} x_i e_i\otimes e_i\in\mathcal D_0$, direct
substitution into \eqref{eq:sharp-diagonal-bilinear}, together with
$\sum_{a=4}^n z_a=0$, gives
\[
\Rtwo(q_1, h)
=
\frac1{\sqrt2}(x_1-x_2)
=
\langle q_1, h\rangle
\]
and
\[
\Rtwo(\eta, h)
=
-\sum_{a=4}^n z_a
\sum_{\substack{b=4\\b\neq a}}^n x_b
=
\sum_{a=4}^n z_ax_a
=
\langle\eta, h\rangle.
\]
Since diagonal and off-diagonal tensors have no mixed curvature terms,
$q_1$ and every nonzero $\eta$ are eigentensors of $\Rtwo$ with
eigenvalue $1$. Such tensors $\eta$ form an $(n-4)$-dimensional
subspace.

The mutually orthogonal eigenspaces found above have total dimension
$n-2$. Since $\dim\mathcal D_0=n-1$, one diagonal direction remains.
Define
\[
q_*:=
(n-3)\sum_{i=1}^3e_i\otimes e_i
-3\sum_{a=4}^ne_a\otimes e_a.
\]
The tensor $q_*$ is traceless. Moreover, a direct computation shows that it is orthogonal to $q_0$, $q_1$, and
all the tensors $\eta$. Hence it spans the remaining one-dimensional
subspace of $\mathcal D_0$.

For any $h=\sum_{i=1}^{n} x_i e_i\otimes e_i\in\mathcal D_0$, using
\eqref{eq:sharp-diagonal-bilinear} and
$\sum_{a=4}^n x_a=-(x_1+x_2+x_3)$, we obtain
\[
\begin{aligned}
\Rtwo(q_*, h)
&=
3(n-3)(x_1+x_2+x_3)
-\frac{2n+3}{3}\sum_{a=4}^n x_a\\
&=
\frac{11n-24}{3}(x_1+x_2+x_3).
\end{aligned}
\]
Moreover,
$\langle q_*, h\rangle=n(x_1+x_2+x_3)$. Therefore,
\[
\Rtwo(q_*, h)
=
\left(\frac{11}{3}-\frac8n\right)
\langle q_*, h\rangle.
\]
Since diagonal and off-diagonal tensors have no mixed curvature terms,
$q_*$ is an eigentensor with eigenvalue
$\frac{11}{3}-\frac8n>1$ for $n\geq4$.

Combining the diagonal and off-diagonal calculations, the four
smallest eigenvalues of $\Rtwo$ are
\[
\lambda_1=-\frac13,\qquad
\lambda_2=\lambda_3=0,\qquad
\lambda_4=1,
\]
and $\lambda_j\geq1$ for every $j\geq4$. Hence
\[
\lambda_1+\lambda_2+\lambda_3+\frac13\lambda_4=0,
\]
so $\Rtwo$ is $\frac{10}{3}$-nonnegative. If
$\frac{10}{3}<\alpha\leq4$, the sum defining $\alpha$-positivity is
$\alpha-\frac{10}{3}>0$. For $\alpha>4$, this sum remains positive
because all the additional eigenvalues are at least $1$. Thus,
$\Rtwo$ is $\alpha$-positive for every
$\alpha\in(\frac{10}{3}, N]$.

Finally, for the four-frame $\{e_1, e_2, e_3, e_4\}$ and the parameter
$\lambda=0$, the left-hand side of \eqref{eq:PIC1} is
\[
R_{1313}+R_{2323}=0.
\]
Thus, the strict PIC1 inequality fails, and $R$ is not PIC1. On the
other hand, the nonnegative statement in
Lemma \ref{lem:main-algebraic} implies that $R$ is NIC1.
\end{proof}

\begin{remark}
The example above shows that $\alpha$-positivity does not imply PIC1
for any $\alpha>\frac{10}{3}$. Therefore, an argument that first deduces PIC1 from $\alpha$-positivity cannot weaken the curvature assumption beyond $\frac{10}{3}$. 
Weaker curvature assumptions can still be treated by other methods; see \cite[Theorem 1.3]{Li22JGA}.
\end{remark}

\medskip

\noindent
\textbf{Acknowledgments}
The first author would like to express sincere gratitude to Professor Ying Zhang for valuable guidance, helpful suggestions, and continuous encouragement. 
The authors would like to thank Professor Xiaolong Li for suggesting the problem and for many helpful discussions.
The research of this paper is partially supported by  NSF of Jiangsu Province Grant No.BK20231309.

\vspace{0.3em}

\noindent
\textbf{AI Disclosure.}
OpenAI models were used as auxiliary tools during the preparation of this manuscript.
All mathematical arguments and proofs were independently verified by the authors, who take full responsibility for the content of the paper.

\bibliographystyle{alpha}
\bibliography{ref}

\end{document}